\documentclass[11pt]{article}
\usepackage{amsmath, amsthm, amssymb, amsfonts, latexsym, amscd}
\usepackage[dvips]{graphicx}
\newtheorem{theorem}{Theorem}[section]

\newtheorem{lemma}[theorem]{Lemma}

\newtheorem{example}[theorem]{Example}
\newtheorem{claim}{Claim}[theorem]

\begin{document}

\title{Revisiting Eisenstein-type criterion over integers}
\author{Akash Jena\footnote{The author is a final year student of the Integrated M.Sc. program in Mathematics. He is a recipient of INSPIRE student fellowship awarded by the Department of Science and Technology, Government of India.} \and Binod Kumar Sahoo}
\date{}
\maketitle

\begin{center}
School of Mathematical Sciences\\
National Institute of Science Education and Research, Bhubaneswar\\
HBNI, P.O.- Jatni, Dist- Khurda, Odisha - 752050, India\\
Emails: akash.jena@niser.ac.in, bksahoo@niser.ac.in
\end{center}

\begin{abstract}
The following result, a consequence of Dumas criterion for irreducibility of polynomials over integers, is generally proved using the notion of Newton diagram:
\begin{quote}
``Let $f(x)$ be a polynomial with integer coefficients and $k$ be a positive integer relatively prime to $\deg f(x)$. Suppose that there exists a prime number $p$ such that the leading coefficient of $f(x)$ is not divisible by $p$, all the remaining coefficients are divisible by $p^k$, and the constant term of $f(x)$ is not divisible by $p^{k+1}$. Then $f(x)$ is irreducible over $\mathbb{Z}$''.
\end{quote}
For $k=1$, this is precisely the Eisenstein criterion. The aim of this article is to give an alternate proof, accessible to the undergraduate students, of this result for $k\in \{2,3,4\}$ using basic divisibility properties of integers.\\

\end{abstract}

\newpage

\section{Introduction}

Let $\mathbb{Z}[x]$ be the ring of polynomials with coefficients from the ring $\mathbb{Z}$ of integers. A nonconstant polynomial $f(x)\in\mathbb{Z}[x]$ is said to be {\it reducible} over $\mathbb{Z}$ if it can be written as a product of two nonconstant polynomials in $\mathbb{Z}[x]$, otherwise, $f(x)$ is called {\it irreducible} over $\mathbb{Z}$. There is no universal criterion which can be applied to determine the reducibility/irreducibility of all the polynomials in $\mathbb{Z}[x]$. However, many criteria exist in the literature each of which give this information for some particular class of polynomials. One such criterion, the so called ``Eisenstein criterion'', is due to Gotthold Eisenstein (1823--1852), a German mathematician. This is perhaps the most well-known criterion which gives a sufficient condition for a polynomial in $\mathbb{Z}[x]$ to be irreducible.
\begin{quote}
{\it Eisenstein criterion.} Let $f(x)$ be a polynomial in $\mathbb{Z}[x]$ of positive degree. Suppose that there exists a prime number $p$ such that the leading coefficient of $f(x)$ is not divisible by $p$, all the remaining coefficients are divisible by $p$, and the constant term is not divisible by $p^2$. Then $f(x)$ is irreducible over $\mathbb{Z}$.
\end{quote}
A polynomial satisfying the conditions of Eisenstein criterion for some prime is called an {\it Eisenstein polynomial}. In practice, it may happen that the original polynomial $f(x)$ is not Eisenstein for any prime, but the criterion is applicable (with respect to some prime) to the polynomial obtained after transforming $f(x)$ by some substitution for $x$. The fact that the polynomial after substitution is irreducible then allows to conclude that the original polynomial itself is irreducible.

As mentioned in \cite[p.49]{sury}, one can reverse the roles of the constant term and the leading coefficient of $f(x)$ to get another version of the Eisenstein criterion. More precisely, {\it if the constant term of $f(x)$ is not divisible by $p$, all the remaining coefficients are divisible by $p$, and the leading coefficient of $f(x)$ is not divisible by $p^2$, then $f(x)$ is irreducible over $\mathbb{Z}$}.

We learn Eisenstein criterion generally at the undergraduate level as a part of our mathematics training. At that time, realizing its power and simplicity, students try to generalize the statement of the criterion and ask the following natural question:
\begin{quote}
Suppose that there exists a prime number $p$ and an integer $k\geq 2$ such that the leading coefficient of $f(x)$ is not divisible by $p$, all the remaining coefficients are divisible by $p^k$, and the constant term is not divisible by $p^{k+1}$. Is $f(x)$ necessarily irreducible over $\mathbb{Z}$?
\end{quote}
The answer is certainly No!. For example, one can have the following factorizations:
\begin{enumerate}
\item[] $x^{2}-p^{2}=(x-p)(x+p)$,\; $x^3-p^{3}=(x-p)(x^2+p x+p^{2})$, etc.
\end{enumerate}
However, the answer could be affirmative if one adds an extra condition connecting $k$ and the degree of $f(x)$, see Theorem \ref{general} below.

\subsection{Dumas Criterion}

The second best known irreducibility criterion based on divisibility of the coefficients by a prime is probably the so called ``Dumas Criterion'', due to Gustave Dumas (1872--1955), a Swiss mathematician. To state this criterion, it is necessary to recall the notion of `Newton diagram' of a polynomial over integers with respect to a given prime number.

Let $p$ be a fixed prime number and let $f(x)\in\mathbb{Z}[x]$ be a polynomial of degree $n\geq 1$.
We refer to \cite[Sec. 2.2.1]{Pra} or \cite[Page 96]{O} for the construction of the Newton diagram of $f(x)$ with respect to $p$. Let
$$f(x)=a_n x^n+a_{n-1}x^{n-1}+\dots +a_1x+a_0, $$
where the constant term $a_0$ is nonzero (otherwise, $f(x)$ would be reducible for $n\geq 2$). Every nonzero coefficient $a_i$ of $f(x)$ can be written in the form
$$a_i=\bar{a}_i p^{\alpha_i},$$
where $\bar{a}_i$ is an integer not divisible by $p$, that is, $\alpha_i$ is the maximum power of $p$ such that $p^{\alpha_i}\mid a_i$. Set
$$X=\{(i,\alpha_i):a_i\neq 0\},$$
call the elements of $X$ as vertices and plot them in the plane. Since $f(x)$ is of positive degree, there are at least two vertices: the {\it initial} vertex $(0,\alpha_0)$ and the {\it terminal} vertex $(n, \alpha_n)$. Note that there is no vertex corresponding to a zero coefficient of $f(x)$. The construction of the {\it Newton diagram of $f(x)$ with respect to $p$} is as follows.

Start with the initial vertex $v_0=(0,\alpha_0)$. Then find the vertex $v_1=(i_1,\alpha_{i_1})$, where $i_1\neq 0$ is the largest integer for which there is no vertex of $X$ below the line through $v_0$ and $v_1$. It may happen that the line segment $v_0 v_1$ joining $v_0$ and $v_1$ contain vertices from $X$ which are different from $v_0$ and $v_1$. Then find the vertex $v_2=(i_2,\alpha_{i_2})$, where $i_2$ ($\neq i_1)$ is the largest integer for which there is no vertex below the line through $v_1$ and $v_2$. Proceed in this way to draw the line segments $v_0v_1$, $v_1v_2$ etc. one by one. The very last line segment is of the form $v_{k-1}v_k$, where $v_k=(n,\alpha_n)$ is the terminal vertex. Then the Newton diagram of $f(x)$ with respect to $p$ consists of the line segments $v_{j-1}v_j$, $1\leq j\leq k$. It has at least one line segment.
We say that a line segment $v_{i-1}v_i$ is {\it simple} if $v_{i-1}$ and $v_i$ are the only points on it with integer coordinates.

We now state the irreducibility criterion by Dumas, a proof of which can be found in \cite[Sec. 2.2]{Pra}. The original proof by Dumas appeared in 1906 in the paper \cite{D}.

\begin{quote}
{\it Dumas criterion.} Let $f(x)\in\mathbb{Z}[x]$ be a polynomial of positive degree. Suppose that there exists a prime $p$ for which the Newton diagram of $f(x)$ consists of exactly one simple line segment. Then $f(x)$ is irreducible over $\mathbb{Z}$.
\end{quote}
Observe that if $p$ satisfies the three conditions of Eisenstein criterion, then the Newton diagram of $f(x)$ with respect to $p$ consists of one simple line segment with end vertices $(0,1)$ and $(n,0)$ and so $f(x)$ is irreducible. Thus Dumas criterion can be considered as a generalization of Eisenstein criterion.

\begin{example}
The Newton diagram of $f(x)=x^4+12$ with respect to $p=2$ consists of one line segment through the initial vertex $(0,2)$ and the terminal vertex $(4,0)$. It contains the point $(2,1)$ with integer coordinates and so Dumas criterion can not be applied with respect to $2$. However, $f(x)$ is Eisenstein for $p=3$ and hence irreducible over $\mathbb{Z}$.
\end{example}

Now let $f(x)=a_nx^n+a_{n-1}x^{n-1}+\dots+a_1 x+a_0\in\mathbb{Z}[x]$. Suppose that there exists a positive integer $k$ and a prime number $p$ such that
\begin{enumerate}
\item[] $p\nmid a_n$, $p^k\mid a_j$ ($0\leq j\leq n-1$) and $p^{k+1}\nmid a_0$.
\end{enumerate}
Then the Newton diagram of $f(x)$ with respect to $p$ consists of exactly one line segment $uv$, where $u=(0,k)$ and $v=(n,0)$. The equation of the line through $u$ and $v$ is $$kX+nY=nk.$$
If $k$ and $n$ are relatively prime, then it can be seen that there is no integer coordinate points on the line segment $uv$ different from $u,v$. So $uv$ is a simple line segment and hence $f(x)$ is irreducible by Dumas criterion. Thus, we have the following result which is related to the question mentioned before.

\begin{theorem}\label{general}
Let $f(x)=a_nx^n+a_{n-1}x^{n-1}+\dots+a_1 x+a_0\in\mathbb{Z}[x]$ and $k$ be a positive integer relatively prime to $n$. Suppose that there exists a prime $p$ such that $p\nmid a_n$, $p^k\mid a_j$ for $0\leq j\leq n-1$ and $p^{k+1}\nmid a_0$. Then $f(x)$ is irreducible over $\mathbb{Z}$.
\end{theorem}

For $k=1$, Theorem \ref{general} is simply Eisenstein criterion. The aim of this article is to give an elementary proof, which is accessible to the undergraduate students, of Theorem \ref{general} for $k\in \{2,3,4\}$ using basic divisibility properties of integers. One can use similar argument for other small values of $k$, but more steps will be involved.
For $k\geq 2$, it can be observed from the Newton diagram of $f(x)$ with respect to $p$ that the condition $p^k\mid a_j$ for $0\leq j\leq n-1$ is much stronger. It can further be relaxed for higher value of $j$. For example, for $k=2$, this condition can be replaced by that $p\mid a_j$ for $j\leq n-1$ and $p^2\mid a_i$ for $0\leq i\leq \lfloor n/2\rfloor$ (see Theorem \ref{main-1}).

\section{For $k=2$}
We start with the following lemma which essentially proves the Eisenstein criterion, but stated in a different way as per our requirement. This result is useful in all the cases of $k\in\{2,3,4\}$.

\begin{lemma} \label{elem-1}
Let $f(x)=a_nx^n+a_{n-1}x^{n-1}+\dots+a_1x+a_0\in \mathbb{Z}[x]$. Suppose that there exists a prime $p$ such that $p\nmid a_n$ and $p\mid a_i$ for $0\leq i\leq n-1$. If $f(x)=g(x)h(x)$ for two nonconstant polynomials $g(x)$, $h(x)$ in $\mathbb{Z}[x]$, then $p$ divides all the coefficients, except the leading ones, of $g(x)$ and $h(x)$. In particular, if $p^2\nmid a_0$, then $f(x)$ is irreducible over $\mathbb{Z}$.
\end{lemma}

\begin{proof}
Let $g(x)=b_kx^k+\dots +b_1x+b_0$ and $h(x)=c_{l}x^{l}+\dots+c_1x+c_0$, where $k,l\geq 1$. We first show that $b_0$ and $c_0$ are divisible by $p$. Since $a_n=b_kc_l$ and $p\nmid a_n$, we have $p\nmid b_k$ and $p\nmid c_l$. Since $a_0=b_0c_0$ and $p\mid a_0$, we have $p\mid b_0$ or $p\mid c_0$. We may assume that $p\mid b_0$. Let $r$, $1\leq r\leq k$, be the smallest integer such that $p\nmid b_r$. Considering the coefficient $a_r$ in $f(x)$, we have
$$a_r=b_rc_0+b_{r-1}c_1+\dots +b_0c_r.$$
Since $p\mid a_r$ and $p\mid b_i$ for $0\leq i\leq r-1$, it follows that $p\mid b_rc_0$. Then $p\mid c_0$ as $p\nmid b_r$.

Now consider $r$ as above and let $s$, $1\leq s\leq l$, be the smallest integer such that $p\nmid c_s$. Considering the coefficient $a_{r+s}$ in $f(x)$, we have
$$a_{r+s}=b_{r+s}c_0+\dots +b_{r+1}c_{s-1}+b_rc_s+b_{r-1}c_{s+1}+\dots +b_0c_{r+s}.$$
Note that $p\nmid b_rc_s$ and all the remaining terms in the above expression of $a_{r+s}$ are divisible by $p$. So $p\nmid a_{r+s}$. Since $p\mid a_i$ for $0\leq i\leq n-1$, we get $a_{r+s}=a_n=a_{k+l}$. So $r=k$ and $s=l$.
\end{proof}

We now prove the following result which is an improved version of Theorem \ref{general} for $k=2$.

\begin{theorem}\label{main-1}
Let $f(x)=a_{n}x^{n}+a_{n-1}x^{n-1}+\dots +a_1x+a_0\in\mathbb{Z}[x]$. Suppose that there exists a prime $p$ such that $p\nmid a_{n}$, $p\mid a_i$ for $i\leq n-1$, $p^2\mid a_j$ for $j\leq \lfloor n/2\rfloor$ and $p^3\nmid a_0$.
Then the following hold:
\begin{enumerate}
\item[(1)] If $n$ is odd, then $f(x)$ is irreducible over $\mathbb{Z}$.
\item[(2)] If $n$ is even, then either $f(x)$ is irreducible over $\mathbb{Z}$, or $f(x)$ is a product of exactly two irreducible polynomials in $\mathbb{Z}[x]$ of equal degree which are Eisenstein with respect to $p$.
\end{enumerate}
\end{theorem}

\begin{proof}
Suppose that $f(x)=g(x)h(x)$ for some nonconstant polynomials $g(x),h(x)$ in $\mathbb{Z}[x]$, where
\begin{enumerate}
\item[]$g(x)=b_rx^r+b_{r-1}x^{r-1}+\dots+b_1x+b_0$,\\
$h(x)=c_{s}x^{s}+c_{s-1}x^{s-1}+\dots+c_1x+c_0$.
\end{enumerate}
Since $p\nmid a_n$, we have $p\nmid b_r$ and $p\nmid c_s$. By Lemma \ref{elem-1}, $b_i$ and $c_j$ are divisible by $p$ for $0\leq i\leq r-1$ and $0\leq j\leq s-1$. Since $p^3\nmid a_0$, we have $p^2\nmid b_0$ and $p^2\nmid c_0$. Thus $g(x)$ and $h(x)$ both are Eisenstein with respect to $p$ and hence irreducible over $\mathbb{Z}$. In order to complete the proof, it is enough to show that $r=s$.

First suppose that $r> s$. We shall get a contradiction by showing that $p\mid c_s$. Considering the coefficient $a_s$ in $f(x)$, we have
$$a_s=b_sc_0+b_{s-1}c_1+\dots +b_1c_{s-1}+b_0c_s.$$
Since $r>s$, each term in the above expression of $a_s$, different from $b_0c_s$, is divisible by $p^2$. Also, $\lfloor n/2\rfloor=\lfloor (r+s)/2\rfloor\geq s$ implies that $p^2\mid a_s$. Then it follows that $p^2\mid b_0c_s$. Since $p^2\nmid b_0$, we get $p\mid c_s$, a contradiction. If $s>r$, then similar argument holds to get a contradiction that $p\mid b_r$.
\end{proof}

We give examples below to show that both the possibilities in Theorem \ref{main-1} may occur for even degree polynomials in $\mathbb{Z}[x]$.

\begin{example}
(1) For any prime $p$, the polynomial $f(x)=x^2+p^2\in\mathbb{Z}[x]$ satisfies the conditions of Theorem \ref{main-1} with respect to $p$. So it is irreducible over $\mathbb{Z}$.

(2) The polynomial $f(x)=x^4+5x^3+25x^2+50x+150$ satisfies the conditions of Theorem \ref{main-1} with $p=5$. But it is reducible over $\mathbb{Z}$, as $f(x)=(x^2+10)(x^2+5x+15)$.
\end{example}

\section{For $k=3$}

The following elementary result is useful for us. We include a proof of it for the sake of completeness.

\begin{lemma}\label{elem}
Let $p$ be a prime and $u,v$ be integers which are not divisible by $p$. If $p\mid xy$ and $p\mid (ux+vy)$ for some integers $x$ and $y$, then $p\mid x$ and $p\mid y$.
\end{lemma}

\begin{proof}
Since $p$ is a prime and $p\mid xy$, we have $p\mid x$ or $p\mid y$. Assume that $p\mid x$. Then $p\mid (ux+vy)$ implies that $p\mid vy$. Then $p\mid y$ as $p\nmid v$.
\end{proof}

\begin{theorem}\label{main-2}
Let $f(x)=a_{n}x^{n}+a_{n-1}x^{n-1}+\dots +a_1x+a_0\in\mathbb{Z}[x]$, where $n$ is not divisible by $3$. Suppose that there exists a prime $p$ such that the leading coefficient $a_n$ is not divisible by $p$, the remaining coefficients are divisible by $p^3$ and the constant term $a_0$ is not divisible by $p^4$.
Then $f(x)$ is irreducible over $\mathbb{Z}$.
\end{theorem}

\begin{proof}
Suppose that $f(x)$ is reducible over $\mathbb{Z}$. Let $f(x)=g(x)h(x)$ for some nonconstant polynomials $g(x),h(x)$ in $\mathbb{Z}[x]$, where
\begin{enumerate}
\item[] $g(x)=b_rx^r+b_{r-1}x^{r-1}+\dots+b_1x+b_0$,\\
$h(x)=c_{s}x^{s}+c_{s-1}x^{s-1}+\dots+c_1x+c_0$.
\end{enumerate}
Since $p\nmid a_n$, $p\nmid b_r$ and $p\nmid c_s$. By Lemma \ref{elem-1}, $p\mid b_i$ for $0\leq i\leq r-1$ and $p\mid c_j$ for $0\leq j\leq s-1$. Since $p^3\mid a_0$ and $p^4\nmid a_0$, either $b_0=up^2$ and $c_0=vp$, or $b_0=up$ and $c_0=vp^2$ for some integers $u, v$ which are not divisible by $p$. Without loss, we may assume that $b_0=up^2$ and $c_0=vp$.

\begin{claim}\label{claim-1}
$r>s$.
\end{claim}

On the contrary, suppose that $s\geq r$. Considering the coefficient $a_r$ in $f(x)$, we have
$$b_r vp+b_{r-1}c_1+\dots +b_1c_{r-1}+c_r up^2=a_r\equiv 0\mod p^2.$$
Since $b_i$ and $c_i$ are divisible by $p$ for $1\leq i\leq r-1$, it follows that $p^2 \mid b_r vp$. Then $p\nmid v$ implies that $p\mid b_r$, a contradiction.

\begin{claim}\label{claim-2}
$p^2\mid b_l$ for $0\leq l\leq s-1$.
\end{claim}

We shall prove by induction on $l$. This is clear for $l=0$, since $b_0=up^2$. So assume that $1\leq l\leq s-1$ and that $p^2\mid b_i$ for $0\leq i\leq l-1$. The coefficient $a_l$ in $f(x)$ is divisible by $p^3$ and so
$$b_lvp+b_{l-1}c_1+\dots +b_1c_{l-1}+c_lup^2\equiv 0 \mod p^3.$$
Using the induction hypothesis and the fact that $p\mid c_i$ for $1\leq i\leq l$, it follows that $p^3\mid b_lvp$. Then $p\nmid v$ implies that $p^2\mid b_l$.

\begin{claim}\label{claim-3}
$r\geq 2s$.
\end{claim}

Suppose that $r\leq 2s-1$. Since the coefficient $a_r$ in $f(x)$ is divisible by $p^2$, we have
$$b_r vp+b_{r-1}c_1+\dots +b_{r-s+1}c_{s-1}+b_{r-s}c_s\equiv 0 \mod p^2.$$
Note that $r-s\leq s-1$ as $r\leq 2s-1$ by our assumption, and so $p^2\mid b_{r-s}$ by Claim \ref{claim-2}. Since $p\mid c_i$ for $1\leq i\leq s-1$ and $p\mid b_j$ for $j\leq r-1$, it follows that $p^2\mid b_r vp$. Then $p\nmid v$ implies that $p\mid b_r$, a contradiction.

\begin{claim}\label{claim-3-1}
$r\geq 2s+1$.
\end{claim}

By Claim \ref{claim-3}, we have $r\geq 2s$. Since $n=r+s$, the hypothesis that $3\nmid n$ implies $r\neq 2s$. So $r\geq 2s+1$.\\

Now the coefficient $a_s=b_svp+b_{s-1}c_1+\dots +b_1c_{s-1}+c_s up^2$ of $x^s$ in $f(x)$ is divisible by $p^3$. Using Claim \ref{claim-2}, it follows that
\begin{equation}\label{sum}
\bar{b}_sv+c_su\equiv 0 \mod p,
\end{equation}
where $b_s=\bar{b}_s p$. Considering the coefficient $a_{2s}$ of $x^{2s}$ in $f(x)$ which is divisible by $p^2$, we have
$$b_{2s}vp+b_{2s-1}c_1+\dots+b_{s+1}c_{s-1}+b_s c_s\equiv 0 \mod p^2.$$
Note that $p\mid b_{j}$ for $j\leq 2s$ as $r\geq 2s+1$. It follows that $b_s c_s$ is divisible by $p^2$ and this gives
\begin{equation}\label{product}
\bar{b}_s c_s \equiv 0 \mod p.
\end{equation}
Then the congruence relations (\ref{sum}), (\ref{product}) and Lemma \ref{elem} together imply that $p\mid c_s$, a final contradiction to our assumption that $f(x)$ is reducible.  This completes the proof.
\end{proof}

\section{For $k=4$}

\begin{theorem}
Let $f(x)=a_{n}x^{n}+a_{n-1}x^{n-1}+\dots +a_1x+a_0\in\mathbb{Z}[x]$, where $n$ and $4$ are relatively prime. Suppose that there exists a prime $p$ such that the leading coefficient $a_n$ is not divisible by $p$, the remaining coefficients are divisible by $p^4$ and the constant term $a_0$ is not divisible by $p^5$.
Then $f(x)$ is irreducible over $\mathbb{Z}$.
\end{theorem}

\begin{proof}
Suppose that $f(x)$ is reducible over $\mathbb{Z}$. Let $f(x)=g(x)h(x)$ for some nonconstant polynomials $g(x),h(x)$ in $\mathbb{Z}[x]$, where
\begin{enumerate}
\item[] $g(x)=b_rx^r+b_{r-1}x^{r-1}+\dots+b_1x+b_0$,\\
$h(x)=c_{s}x^{s}+c_{s-1}x^{s-1}+\dots+c_1x+c_0$.
\end{enumerate}
Then $p\nmid b_r$ and $p\nmid c_s$, since $p\nmid a_n$. By Lemma \ref{elem-1}, $p\mid b_i$ for $0\leq i\leq r-1$ and $p\mid c_j$ for $0\leq j\leq s-1$. Since $p^4\mid a_0$ and $p^5\nmid a_0$, we have the following cases for some integers $u, v$ which are not divisible by $p$:
\begin{enumerate}
\item[] (1) either $b_0=up^3$ and $c_0=vp$, or $b_0=up$ and $c_0=vp^3$.
\item[] (2) $b_0=up^2$ and $c_0=vp^2$.
\end{enumerate}

{\bf Case-(1)}. Without loss, we may assume that $b_0=up^3$ and $c_0=vp$. Applying the argument as in the proof of Claim \ref{claim-1}, we get $r >s$. Then applying similar arguments as in the proof of Claims \ref{claim-2} and \ref{claim-3}, we have the following facts:
$$p^3\mid b_l\text{ for }0\leq l\leq s-1 \text{ and } r\geq 2s.$$
The coefficient $a_s=b_svp+b_{s-1}c_1+\dots +b_1c_{s-1}+c_s up^3$ in $f(x)$ is divisible by $p^4$. It follows that $p^2\mid b_s$ and that
\begin{equation}\label{sum-4-1}
\bar{b}_sv+c_su\equiv 0 \mod p,
\end{equation}
where $b_s=\bar{b}_s p^2$. Since $p\nmid u$ and $p\nmid c_s$, (\ref{sum-4-1}) implies that $p\nmid \bar{b}_s$.

\begin{claim}\label{claim-4}
$p^2\mid b_{s+t}$ for $0\leq t\leq s-1$.
\end{claim}

We prove this by induction on $t$. For $t=0$, we have obtained above that $p^2\mid b_s$. So assume that $1\leq t\leq s-1$ and that $p^2\mid b_{s+i}$ for $0\leq i\leq t-1$. We have
$$a_{s+t}=b_{s+t}vp+b_{s+t-1}c_1+\dots+b_{t+1}c_{s-1}+b_{t}c_s.$$
Note that $b_t$ is divisible by $p^3$ as $t\leq s-1$. Using the induction hypotheses, it follows that all the terms, different from the first one, in the above expression of $a_{s+t}$ are divisible by $p^3$. Since $p^3\mid a_{s+t}$, we get $p^3\mid b_{s+t}vp$ and so $p^2\mid b_{s+t}$ as $p\nmid v$.

\begin{claim}
$r\geq 3s+1$.
\end{claim}

First suppose that $r\leq 3s-1$. Considering the coefficient $a_r$ of $x^r$ in $f(x)$ which is divisible by $p^2$, we have
$$b_r vp+b_{r-1}c_1+\dots +b_{r-s+1}c_{s-1}+b_{r-s}c_s\equiv 0 \mod p^2.$$
Note that $r-s\leq 2s-1$ as $r\leq 3s-1$ by our assumption, and so $p^2\mid b_{r-s}$ by Claim \ref{claim-4}. Since $p\mid c_i$ for $1\leq i\leq s-1$ and $p\mid b_j$ for $j\leq r-1$, it follows that $p^2\mid b_r vp$. Then $p\nmid v$ implies $p\mid b_r$, a contradiction. Thus $r\geq 3s$. Since $n=r+s$, and $4$ and $n$ are relatively prime, we get $r\geq 3s+1$.\\

The coefficient $a_{2s}=b_{2s}vp+b_{2s-1}c_1+\dots+b_{s+1}c_{s-1}+b_{s} c_s$ of $x^{2s}$ in $f(x)$ is divisible by $p^3$. Using Claim \ref{claim-4}, it follows that
\begin{equation}\label{sum-4-2}
\bar{b}_{2s}v+\bar{b}_sc_s\equiv 0\mod p,
\end{equation}
where $b_{2s}=\bar{b}_{2s}p$ and $\bar{b}_s$ is as before. Considering the coefficient $a_{3s}$ in $f(x)$ which is divisible by $p^2$, we have
$$b_{3s}vp+b_{3s-1}c_1+\dots+b_{2s+1}c_{s-1}+b_{2s} c_s\equiv 0 \mod p^2.$$
Note that $p\mid b_{j}$ for $j\leq 3s$ as $r\geq 3s+1$. It follows that $b_{2s} c_s$ is divisible by $p^2$ and this gives
\begin{equation}\label{product-4-1}
\bar{b}_{2s} c_s \equiv 0 \mod p.
\end{equation}
Since $p\nmid v$ and $p\nmid \bar{b}_s$, Lemma \ref{elem} applying to the congruence relations (\ref{sum-4-2}) and (\ref{product-4-1}) gives $p\mid c_s$, a contradiction.  This completes the proof of Case-(1).\\

{\bf Case-(2)}. Here $b_0=up^2$ and $c_0=vp^2$. Without loss, we may assume that $r\geq s$. Since $4$ and $n$ are relatively prime, we must have $r>s$.

\begin{claim}
$b_l$ and $c_l$ are divisible by $p^2$ for $0\leq l\leq \lfloor (s-1)/2 \rfloor$.
\end{claim}

We shall prove by induction on $l$. This is clear for $l=0$, since $b_0=up^2$ and $c_0=vp^2$. So assume that $1\leq l\leq \lfloor (s-1)/2 \rfloor$ and that $b_i,c_i$ are divisible by $p^2$ for $0\leq i\leq l-1$. The coefficient $a_l$ in $f(x)$ is divisible by $p^4$ and so
$$b_lvp^2+b_{l-1}c_1+\dots +b_1c_{l-1}+c_lup^2\equiv 0 \mod p^4.$$
Using the induction hypothesis, we get
\begin{equation}\label{sum-4-3}
\bar{b}_{l}v+\bar{c}_l u\equiv 0\mod p,
\end{equation}
where $b_l=\bar{b}_l p$ and $c_l=\bar{c}_l p$. For the coefficient $a_{2l}$ in $f(x)$, we have
$$a_{2l}=b_{2l}vp^2+b_{2l-1}c_1+\dots +b_{l+1}c_{l-1}+b_l c_l+b_{l-1}c_{l+1}+\dots+b_1c_{2l-1}+c_{2l}up^2.$$
Since $a_{2l}\equiv 0 \mod p^3$, again using the induction hypothesis, it follows that
\begin{equation}\label{product-4-2}
\bar{b}_{l}\bar{c}_l \equiv 0\mod p.
\end{equation}
Then (\ref{sum-4-3}), (\ref{product-4-2}) and Lemma \ref{elem} together imply that $p\mid\bar{b}_l$ and $p\mid \bar{c}_l$ and so the claim follows.\\

If $s$ is odd, say $s=2k+1$ for some $k$, then $\lfloor (s-1)/2 \rfloor = k$. Considering the coefficient $a_{s}=a_{2k+1}$ in $f(x)$, we have
$$b_{2k+1}vp^2+b_{2k}c_1+\dots +b_{k+1}c_{k}+b_{k}c_{k+1}+\dots +b_1c_{2k}+c_{2k+1} up^2 \equiv 0\mod p^3.$$
This gives $p\mid c_{2k+1}$, that is, $p\mid c_s$, a contradiction.

If $s$ is even, say $s=2k$ for some $k$, then $\lfloor (s-1)/2 \rfloor = k-1$. For the coefficient $a_s=a_{2k}$ in $f(x)$, we have
$$a_{2k}=b_{2k}vp^2+b_{2k-1}c_1+\dots+b_{k+1}c_{k-1}+b_{k}c_k+b_{k-1}c_{k+1}+\dots +b_{1}c_{2k-1}+c_{2k}up^2.$$
Since $r\geq s+1= 2k+1$ and $a_s=a_{2k}\equiv 0\mod p^3$, it follows that
\begin{equation}\label{sum-4-4}
\bar{b}_k\bar{c}_k+c_{2k}u\equiv 0\mod p,
\end{equation}
where $b_k=\bar{b}_k p$ and $c_k=\bar{c}_k p$. Now, for the coefficient $a_{3k}$ in $f(x)$, we have
$$a_{3k}=b_{3k}vp^2+b_{3k-1}c_1+\dots+b_{2k+1}c_{k-1}+b_{2k}c_k+\dots +b_{k+1}c_{2k-1}+b_kc_{2k}.$$
Each term, different from the last one, in the above expression is divisible by $p^2$. Since $a_{3k}\equiv 0\mod p^2$, we get $p^2\mid b_k c_{2k}$ and so
\begin{equation}\label{product-4-3}
\bar{b}_kc_{2k}\equiv 0\mod p.
\end{equation}
Then, as $p\nmid u$, the congruence relations (\ref{sum-4-4}) and (\ref{product-4-3}) give $p\mid c_{2k}$, that is, $p\mid c_s$, a contradiction. This completes the proof.
\end{proof}

\end{document}